\documentclass[11pt]{article}
\usepackage[utf8]{inputenc}
\usepackage{amsmath}
\usepackage{graphicx}
\usepackage[colorinlistoftodos]{todonotes}
\usepackage{amsthm}
\usepackage{hyperref}
\hypersetup{colorlinks,allcolors=teal}
\usepackage{amssymb}
\usepackage{amsfonts}
\usepackage{mathtools}
\usepackage{marginnote}
\usepackage[T1]{fontenc}
\usepackage{soul}
\usepackage{tikz}
\usepackage{tikz-cd}
\usepackage{caption}
\usepackage{setspace}
\usepackage{enumitem}
\usepackage{authblk}
\usepackage{geometry}
\usetikzlibrary{arrows.meta}
\usetikzlibrary{decorations.markings}

\newtheorem{thm}{Theorem}[section]
\newtheorem{lem}[thm]{Lemma}

\theoremstyle{definition}

\renewcommand{\qedsymbol}{$\blacksquare$}

\newcommand\ULobstacle{}
\def\ULobstacle[#1](#2:#3){
    \draw[draw=#1, thick, fill=#1] (#2-0.45, #3+0.45) -- (#2+0.45, #3-0.45);
}

\newcommand\URobstacle{}
\def\URobstacle[#1](#2:#3){
    \draw[draw=#1, thick, fill=#1] (#2+0.45, #3+0.45) -- (#2-0.45, #3-0.45);
}

\begin{document}
\reversemarginpar

\title{The Manhattan and Lorentz Mirror Models - A result on the Cylinder with low density of mirrors}
\author{Kieran Ryan \footnote{Supported in part by the European Research Council starting grant 639305 (SPECTRUM)}}
\affil{\href{mailto:kieran.ryan@qmul.ac.uk}{kieran.ryan@qmul.ac.uk}\\
Department of Mathematics, Queen Mary University of London, Mile End Road, London E1 4NS}
            
\date{}
\maketitle

\begin{abstract}
    We study the Manhattan and Lorentz Mirror models on an infinite cylinder of finite even width $n$, with the mirror probability $p$ satisfying $p<Cn^{-1}$, $C$ a constant. We use the Brauer and Walled Brauer algebras to show that the maximum height along the cylinder reached by a walker is order $p^{-2}$.
\end{abstract}

\section{Introduction}
The Manhattan and Lorentz mirror models \cite{cardymanhattan03}, \cite{kozma2013lower}, are two very similar models, each describing a random walk on the $\mathbb{Z}^2$ lattice. Let $0 \le p \le 1$. The walker is a particle of light which bounces off mirrors placed at each vertex at $45^{\circ}$, independently with probability $p$. For the Lorentz mirror model, the orientation of the mirror (NW or NE) is chosen independently with probability $\frac{1}{2}$. For the Manhattan model, the lattice is a priori given Manhattan directions (see Figure \ref{fig:1and2:modelexamples}), and the orientation of the mirror is determined by its location (ie. a NW mirror if the sum of the point's coordinates is odd, and NE if the sum is even), so that the walker always follows the directions of the lattice. The main questions of interest in both models are whether the paths remain bounded or not, and the nature of the motion of the walker.

We study the models on an infinite cylinder $\mathbb{Z} \times C_{n}$ of finite \textit{even} width $n$. Note that on the cylinder, paths are bounded with probability 1 - indeed, there is a positive probability that a horizontal row is filled with mirrors such that no path can pass the row. We are interested in how the length of the paths vary with $p$. We are not aware of any results of this kind which are not exponential in $p^{-1}$. The result of this paper, Theorem \ref{thm:1.1mainthm}, shows that for both models, when $p \le Cn^{-1}$, $C$ a constant, the maximal vertical distance travelled by a path on the $n$-cylinder is order $p^{-2}$. We wonder whether this is true for all $p$.

We observe an underlying algebraic structure (valid for any value of $p$), that the models on the cylinder can be thought of as Markov chains on the Brauer algebra (in the mirror case), or its subalgebra the Walled Brauer algebra (in the Manhattan case). We suggest that the models' association with different algebraic structures reflects their different behaviours. A third model, on the L-lattice (see \cite{cardyquantumclasslocal02}) is solved using percolation, but can be similarly thought of as a Markov chain on the (extended) Temperley Lieb algebra.\\

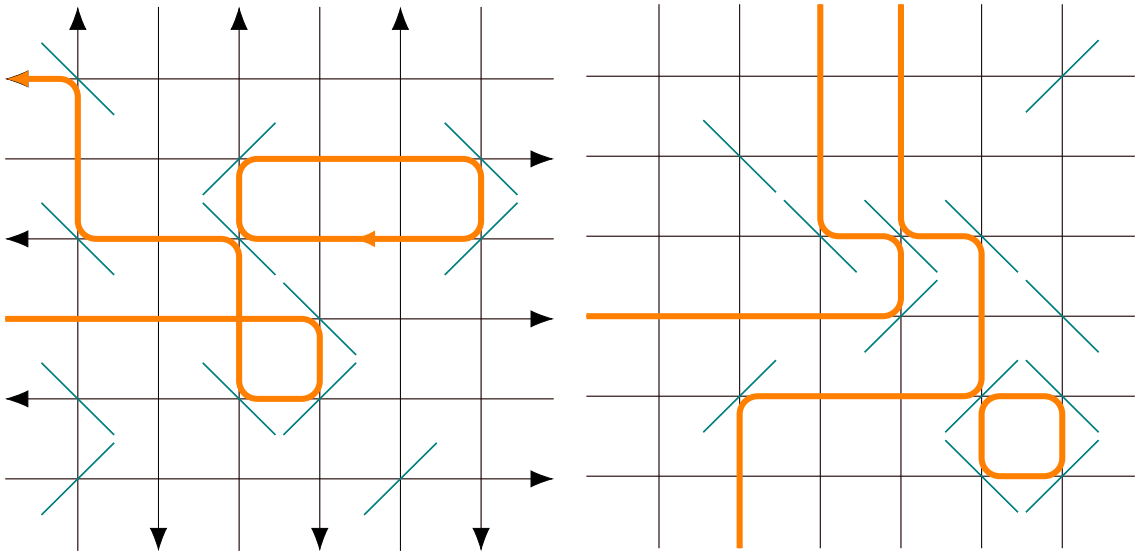
\begin{figure}[h]
    
\centering
\resizebox{1.0\textwidth}{!}{%
\begin{tikzpicture}[scale=1.4, every node/.style={transform shape}]
    
    \draw[step=1cm,pink,very thin] (-0.9,-0.9) grid (5.9,5.9);
    \draw[-{Latex[length=0.4cm]}] (1,5.9) -- (1,-0.9);
    \draw[-{Latex[length=0.4cm]}] (3,5.9) -- (3,-0.9);
    \draw[-{Latex[length=0.4cm]}] (5,5.9) -- (5,-0.9);
    \draw[-{Latex[length=0.4cm]}] (0,-0.9) -- (0,5.9);
    \draw[-{Latex[length=0.4cm]}] (2,-0.9) -- (2,5.9);
    \draw[-{Latex[length=0.4cm]}] (4,-0.9) -- (4,5.9);
    \draw[-{Latex[length=0.4cm]}] (-0.9,4) -- (5.9,4);
    \draw[-{Latex[length=0.4cm]}] (-0.9,2) -- (5.9,2);
    \draw[-{Latex[length=0.4cm]}] (-0.9,0) -- (5.9,0);
    \draw[-{Latex[length=0.4cm]}] (5.9,1) -- (-0.9,1);
    \draw[-{Latex[length=0.4cm]}] (5.9,3) -- (-0.9,3);
    \draw[-{Latex[length=0.4cm]}] (5.9,5) -- (-0.9,5);
    
    \URobstacle[teal](5:3);
    \URobstacle[teal](2:4);
    \URobstacle[teal](0:0);
    \ULobstacle[teal](0:5);
    \URobstacle[teal](3:1);
    \ULobstacle[teal](2:3);
    \ULobstacle[teal](0:1);
    \ULobstacle[teal](3:2);
    \ULobstacle[teal](2:1);
    \URobstacle[teal](4:0);
    \ULobstacle[teal](0:3);
    \ULobstacle[teal](5:4);

    \draw[-{Latex[length=0.4cm]}, rounded corners=9pt, line width=3pt, orange](-0.9, 2) -- (3, 2) -- (3, 1) -- (2, 1) -- (2, 3) -- (0, 3) -- (0, 5) -- (-0.9, 5);
    \draw[-{Latex[length=0.4cm]}, rounded corners=9pt, line width=3pt, orange](3.6, 3) -- (2, 3) -- (2, 4) -- (5, 4) -- (5, 3) -- (3.4, 3);
    
    \end{tikzpicture}
    \quad
    \begin{tikzpicture}[scale=1.4, every node/.style={transform shape}]
    
    \draw[step=1cm,pink,very thin] (-0.9,-0.9) grid (5.9,5.9);
    \draw[-] (1,5.9) -- (1,-0.9);
    \draw[-] (3,5.9) -- (3,-0.9);
    \draw[-] (5,5.9) -- (5,-0.9);
    \draw[-] (0,-0.9) -- (0,5.9);
    \draw[-] (2,-0.9) -- (2,5.9);
    \draw[-] (4,-0.9) -- (4,5.9);
    \draw[-] (-0.9,4) -- (5.9,4);
    \draw[-] (-0.9,2) -- (5.9,2);
    \draw[-] (-0.9,0) -- (5.9,0);
    \draw[-] (5.9,1) -- (-0.9,1);
    \draw[-] (5.9,3) -- (-0.9,3);
    \draw[-] (5.9,5) -- (-0.9,5);
    
    \ULobstacle[teal](1:4);
    \URobstacle[teal](5:5);
    \ULobstacle[teal](2:3);
    \ULobstacle[teal](3:3);
    \ULobstacle[teal](4:3);
    \URobstacle[teal](3:2);
    \ULobstacle[teal](5:2);
    \URobstacle[teal](1:1);
    \URobstacle[teal](4:1);
    \ULobstacle[teal](4:0);
    \ULobstacle[teal](5:1);
    \URobstacle[teal](5:0);
    
    \draw[-, rounded corners=9pt, line width=3pt, orange](2,5.9) -- (2,3) -- (3,3) -- (3,2) -- (-0.9, 2);
    \draw[-, rounded corners=9pt, line width=3pt, orange](3,5.9) -- (3,3) -- (4,3) -- (4,1) -- (1,1) -- (1, -0.9);
    \draw[-, rounded corners=9pt, line width=3pt, orange](4.5,0) -- (5,0) -- (5,1) -- (4,1) -- (4,0) -- (4.5,0);
    
    \end{tikzpicture}
    
}%
    
    \caption{Examples of the Manhattan model (left) and Mirror model (right), with mirrors in blue, and a few paths highlighted in orange. Note that the orientation of a mirror in the Manhattan case is determined by the Manhattan directions of the lattice.}
    \label{fig:1and2:modelexamples}
\end{figure}

Let us recap the existing results on both models (which are on $\mathbb{Z}^2$, unless specified). The Mirror model was introduced by Ruijgrok and Cohen \cite{ruijgrok&cohen1988} as a lattice version of the Ehrenfest wind-tree model. Grimmett \cite{grimmettbook} proved that on $\mathbb{Z}^2$, if $p=1$, then the walker's path is bounded with probability 1. It is conjectured that this is also true for $0 < p \le 1$. This is supported by numerical simulations, for example, in \cite{ziffkongcohen1991}. More recently, Kozma and Sidoravicius \cite{kozma2013lower} showed that, for any $0 < p \le 1$, the probability the walker reaches the boundary of the $n$-box $[-n,n]^2$ is at least $\frac{1}{2n+1}$. To obtain this result, they study the model on an infinte cylinder of finite \textit{odd} width, where there is deterministically always an infinite path. The Manhattan model cannot be neatly defined on a cylinder of odd width, so this method cannot be applied. 

The Mirror model on the cylinder (under the name the O(1) loop model) has been studied using the Brauer algebra before, in several papers relating to a conjecture (and variations thereof) by Razumov and Stroganov \cite{razumov2001combinatorial}, \cite{zinnjustin2004limdistn}, \cite{francesco2004razumovstroganovsumrule}, which gives the entries of the limiting distribution in terms of combinatorial objects such as alternating sign matrices. A generalised mirror model (the $O(q)$ loop model), where the distribution on configurations is weighted by $q^{\#loops}$, $q \in \mathbb{C}$, is studied in \cite{nienhuis1993weighted}; this is the model on the Brauer algebra with parameter $q$, $\mathbb{B}_{n,q}$. In these papers, the requirement of a Yang-Baxter equation restricts the permissible values of the parameters - in our specific setup, only $p=\frac{8}{9}$ qualifies (see the end of \cite{nienhuis1993weighted}).

The Manhattan model shares features of quantum disordered systems. The model was introduced by Beamond, Cardy and Owczarek \cite{cardyquantumclasslocal02}, in close relation to a quantum network model on the Manhattan lattice. The quantum model has random $Sp(2) = SU(2)$ matrices on each edge of the lattice, and the classical model arises on averaging over this disorder. In most classical models in two dimensions, localisation (bounded paths) is not observed, whereas in the Manhattan (and Mirror) model, it is expected (see below). It is not clear if the mirror model has a similar explicit relationship with a quantum model. For more detail on the connection to quantum models, see Spencer's review \cite{spencer2012duality}. 

An argument from \cite{cardyquantumclasslocal02} for tackling the Manhattan model uses percolation. The placement of the mirrors is exactly a Bernoulli percolation on the edges of $\mathbb{Z}^2$, rotated $45^{\circ}$ and scaled. The walker's path stays within $\frac{1}{\sqrt{2}}$ of its closest dual cluster (see Figure 3 ***). The dual clusters are finite with probability 1 for $p \ge \frac{1}{2}$, so so are the Manhattan paths. 

For $p > \frac{1}{2}$, the probability that two points are in the same dual cluster decays exponentially in the distance, which gives the same for connection by a Manhattan path. This is markedly different from the Mirror model. For $p = \frac{1}{2}$, the dual clusters' connection probability decays slower than $\frac{1}{n}$ in the distance (see, for example, \cite{grimmettchapter}), so we cannot obtain exponential decay in Manhattan connection probabilities (although this may well still be true). For $p < \frac{1}{2}$, this argument is wholly inconclusive, since dual clusters are almost surely infinite. Numerical simulations in \cite{cardymanhattan03} indicate that paths are finite, with exponential decay in connection probabilities. Clearly for $p=0$, the paths escape in straight lines to infinity. \\


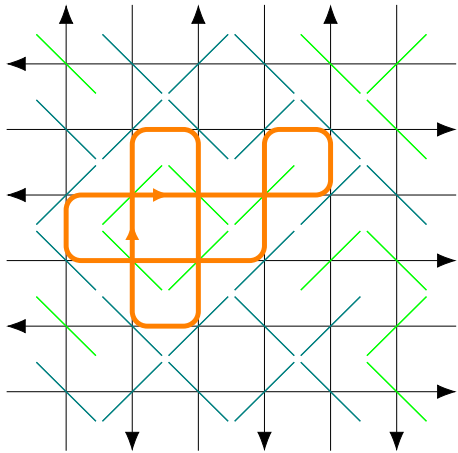
\begin{figure}[h]
    \centering
    \resizebox{0.4\textwidth}{!}{%
    \begin{tikzpicture}[scale=1.4, every node/.style={transform shape}]
    
    \draw[-{Latex[length=0.4cm]}] (1,5.9) -- (1,-0.9);
    \draw[-{Latex[length=0.4cm]}] (3,5.9) -- (3,-0.9);
    \draw[-{Latex[length=0.4cm]}] (5,5.9) -- (5,-0.9);
    \draw[-{Latex[length=0.4cm]}] (0,-0.9) -- (0,5.9);
    \draw[-{Latex[length=0.4cm]}] (2,-0.9) -- (2,5.9);
    \draw[-{Latex[length=0.4cm]}] (4,-0.9) -- (4,5.9);
    \draw[-{Latex[length=0.4cm]}] (-0.9,4) -- (5.9,4);
    \draw[-{Latex[length=0.4cm]}] (-0.9,2) -- (5.9,2);
    \draw[-{Latex[length=0.4cm]}] (-0.9,0) -- (5.9,0);
    \draw[-{Latex[length=0.4cm]}] (5.9,1) -- (-0.9,1);
    \draw[-{Latex[length=0.4cm]}] (5.9,3) -- (-0.9,3);
    \draw[-{Latex[length=0.4cm]}] (5.9,5) -- (-0.9,5);
    
    \ULobstacle[teal](0:0);
    \URobstacle[teal](1:0);
    \ULobstacle[teal](2:0);
    \URobstacle[teal](3:0);
    \ULobstacle[teal](4:0);
    \ULobstacle[teal](1:1);
    \URobstacle[teal](2:1);
    \ULobstacle[teal](3:1);
    \URobstacle[teal](4:1);
    \ULobstacle[teal](0:2);
    \URobstacle[teal](3:2);
    \URobstacle[teal](0:3);
    \URobstacle[teal](4:3);
    \ULobstacle[teal](5:3);
    \ULobstacle[teal](0:4);
    \URobstacle[teal](1:4);
    \ULobstacle[teal](2:4);
    \URobstacle[teal](3:4);
    \ULobstacle[teal](4:4);
    \ULobstacle[teal](1:5);
    \URobstacle[teal](2:5);
    \ULobstacle[teal](3:5);
    \ULobstacle[green](5:0);
    \ULobstacle[green](0:1);
    \URobstacle[green](5:1);
    \ULobstacle[green](1:2);
    \URobstacle[green](2:2);
    \URobstacle[green](4:2);
    \ULobstacle[green](5:2);
    \URobstacle[green](1:3);
    \ULobstacle[green](2:3);
    \URobstacle[green](3:3);
    \ULobstacle[green](5:4);
    \ULobstacle[green](0:5);
    \ULobstacle[green](4:5);
    \URobstacle[green](5:5);
    
    \draw[-{Latex[length=0.4cm]}, rounded corners=9pt, line width=3pt, orange](1,2.4) -- (1,4) -- (2,4) -- (2,1) -- (1,1) -- (1,2.6);
    \draw[-{Latex[length=0.4cm]}, rounded corners=9pt, line width=3pt, orange](1.4,3) -- (4,3) -- (4,4) -- (3,4) -- (3,2) -- (0,2) -- (0,3) -- (1.6,3);
    
    \end{tikzpicture}
    }%

    \caption{The mirrors (in blue) in the Manhattan model as edges in Bernoulli percolation. The green edges form the dual clusters. The two paths shown are restricted to stay within $\frac{1}{\sqrt{2}}$ of one dual cluster.}
    \label{fig:3:mhtnwalkerclosetodualcluster}
\end{figure}

Let us now state our result more precisely. Consider the models on the $n$-cylinder $\mathbb{Z} \times C_{n}$, $n$ even. We label $s_t$ the horizontal row $\{ t \} \times C_{n}$ - the "$t^{th}$ street". For the Mirror model, let $V^{mir}_{\frac{n}{2}}$ be the random variable given by the smallest $t$ such that $s_t$ has no path connecting it to the first street, $s_1$. In other words, the highest street a path from $s_1$ reaches is exactly $V^{mir}_{\frac{n}{2}} -1$. Let $V^{mat}_{\frac{n}{2}}$ be defined identically with the Manhattan model. 
\begin{thm}\label{thm:1.1mainthm}
    Let $*$ represent $mat$ or $mir$.  
    \begin{enumerate}[label=\alph*)]
        \item Let $p\le Cn^{-1}$, $C>0$ a constant. For all $\alpha>0$, 
    \begin{equation*}
        \mathbb{P}[V^{*}_{\frac{n}{2}} \ge \alpha p^{-2}] 
        \le 
        2A_{*}e^{-\frac{1}{8e^C}  \alpha},
    \end{equation*}
    where $A_{mir} = cosh(\pi)$, and $A_{mat} = \frac{sinh(\pi)}{\pi}$.
    
        \item For any $p \le \frac{1}{2}$ (not necessarily constrained by $p<Cn^{-1}$), and for all $\alpha>0$,
        \begin{equation*}
            \mathbb{P} \left[ V^{*}_{\frac{n}{2}} \le \alpha p^{-2} \right] 
            \le 2\alpha.
        \end{equation*}
    \end{enumerate}
\end{thm}

Let us remark on our method from proving part $a)$. For all $p \le Cn^{-1}$, and $n$ large, the probability of obstacles is small, and in particular, the probability that each street $s_k$ has at most two obstacles is large. We show that the model is not changed too much if we actually condition on each $s_k$ having at most two obstacles. This conditioning simplifies the model greatly, in essence removing the cylindrical geometry, making the interactions on each street mean-field, and allows us to do explicit computations. 
Part $b)$ is more straightforward; it is proved by coupling $V_{\frac{n}{2}}^*$ with a geometric random variable $G$ with parameter $p^2$.

In section \ref{section:brauer}, we give key definitions, including the Brauer and Walled Brauer algebras. In section \ref{section:results} we study the model assuming at most two obstacles per street, and obtain the results needed to prove Theorem \ref{thm:1.1mainthm}.\\


    
    

\section{Definitions, and the Brauer algebra}\label{section:brauer}

Let us define the algebraic structures and notation we will use. The \textit{Brauer algebra} $\mathbb{B}_{n,1}$ (see, for example, \cite{brauer}, \cite{wenzl}, \cite{brown}, \cite{cox2007blocks}) is the (formal) complex span of the set of pairings of $2n$ vertices. We think of pairings as graphs, which we will call \textit{diagrams}, with each vertex having degree exactly 1. We arrange the vertices in two horizontal rows, labelling the upper row (the northern vertices) $1^+, 2^+, \dots, n^+$, and the lower (southern) $1^-, \dots, n^-$. We call an edge connecting two northern vertices (or two southern) a \textit{bar}. The number of bars in the north and south is always the same, and we refer to either simply as the number of bars in the diagram. We call an edge connecting a northern and southern vertex a \textit{NS path}.

Multiplication of two diagrams is given by concatenation. If $b, c$ are two diagrams, we align the northern vertices of $b$ with the southern of $c$, and the result is obtained by removing these middle vertices. See Figure \ref{fig:4:multiplicationofdiagramsexample}. This defines $\mathbb{B}_{n,1}$ as an algebra. In general there is a multiplicative parameter $q$, for any $q \in \mathbb{C}$, but for our purposes we only need $q = 1$, which gives the multiplication described above. See \cite{brauer} for more detail.

\begin{figure}[h]
    \centering
    \resizebox{0.95\textwidth}{!}{%
    \begin{tikzpicture}[scale=1]

    \draw[-, rounded corners=10pt, dashed, thick](2, 1) -- (2, 1.5);
    \draw[-, rounded corners=10pt, dashed, thick](0, 1) -- (0, 1.5);
    \draw[-, rounded corners=10pt, dashed, thick](1,1) -- (1, 1.5);
    \draw[-, rounded corners=10pt, dashed, thick](3, 1) -- (3, 1.5);
    \draw[-, rounded corners=10pt, dashed, thick](5,1) -- (5, 1.5);
    \draw[-, rounded corners=10pt, dashed, thick](4,1) -- (4, 1.5);

    \draw[font=\large] (-1, 0.5) node {$b_2$};
    \draw[-, rounded corners=10pt, line width=2pt, orange](1, 1) -- (3, 0);
    \draw[-, rounded corners=10pt, line width=2pt, orange](0, 1) -- (0,0);
    \draw[-, rounded corners=10pt, line width=2pt, orange](2, 1) -- (3, 1);
    \draw[-, rounded corners=10pt, line width=2pt, orange](2, 0) -- (1, 0);
    \draw[-, rounded corners=10pt, line width=2pt, orange](4,0) -- (4, 1);
    \draw[-, rounded corners=10pt, line width=2pt, orange](5,0) -- (5, 1);
    
    \draw[thick, fill=white] (0, 1) circle (0.1cm);
    \draw[thick, fill=white] (1, 1) circle (0.1cm);
    \draw[thick, fill=white] (2, 1) circle (0.1cm);
    \draw[thick, fill=white] (3, 1) circle (0.1cm);
    \draw[thick, fill=white] (4, 1) circle (0.1cm);
    \draw[thick, fill=white] (5, 1) circle (0.1cm);
    \draw[thick, fill=white] (0, 0) circle (0.1cm);
    \draw[thick, fill=white] (1, 0) circle (0.1cm);
    \draw[thick, fill=white] (2, 0) circle (0.1cm);
    \draw[thick, fill=white] (3, 0) circle (0.1cm);
    \draw[thick, fill=white] (4, 0) circle (0.1cm);
    \draw[thick, fill=white] (5, 0) circle (0.1cm);

    \draw[font=\large] (-1, 2) node {$b_1$};
    \draw[-, rounded corners=10pt, line width=2pt, orange](2, 2.5) -- (3, 2.5);
    \draw[-, rounded corners=10pt, line width=2pt, orange](0, 2.5) -- (0, 1.5);
    \draw[-, rounded corners=10pt, line width=2pt, orange](1,2.5) -- (1, 1.5);
    \draw[-, rounded corners=10pt, line width=2pt, orange](2, 1.5) -- (3, 1.5);
    \draw[-, rounded corners=10pt, line width=2pt, orange](5,2.5) -- (5, 1.5);
    \draw[-, rounded corners=10pt, line width=2pt, orange](4,2.5) -- (4, 1.5);
    
    \draw[thick, fill=white] (0, 1.5) circle (0.1cm);
    \draw[thick, fill=white] (1, 1.5) circle (0.1cm);
    \draw[thick, fill=white] (2, 1.5) circle (0.1cm);
    \draw[thick, fill=white] (3, 1.5) circle (0.1cm);
    \draw[thick, fill=white] (4, 1.5) circle (0.1cm);
    \draw[thick, fill=white] (5, 1.5) circle (0.1cm);
    \draw[thick, fill=white] (0, 2.5) circle (0.1cm);
    \draw[thick, fill=white] (1, 2.5) circle (0.1cm);
    \draw[thick, fill=white] (2, 2.5) circle (0.1cm);
    \draw[thick, fill=white] (3, 2.5) circle (0.1cm);
    \draw[thick, fill=white] (4, 2.5) circle (0.1cm);
    \draw[thick, fill=white] (5, 2.5) circle (0.1cm);

    \draw[-, rounded corners=10pt, thick](5.65,1.3) -- (5.95, 1.3);
    \draw[-, rounded corners=10pt, thick](5.65,1.2) -- (5.95, 1.2);

    \draw[font=\large] (13.3, 1.25) node {$b_1 b_2$};
    \draw[-, rounded corners=10pt, line width=2pt, orange](10, 1.75) -- (9, 1.75);
    \draw[-, rounded corners=10pt, line width=2pt, orange](7, 1.75) -- (7, 0.75);
    \draw[-, rounded corners=10pt, line width=2pt, orange](8, 1.75) -- (10, 0.75);
    \draw[-, rounded corners=10pt, line width=2pt, orange](8, 0.75) -- (9, 0.75);
    \draw[-, rounded corners=10pt, line width=2pt, orange](11,1.75) -- (11, 0.75);
    \draw[-, rounded corners=10pt, line width=2pt, orange](12,0.75) -- (12, 1.75);
    
    \draw[thick, fill=white] (7, 0.75) circle (0.1cm);
    \draw[thick, fill=white] (8, 0.75) circle (0.1cm);
    \draw[thick, fill=white] (9, 0.75) circle (0.1cm);
    \draw[thick, fill=white] (10, 0.75) circle (0.1cm);
    \draw[thick, fill=white] (11, 0.75) circle (0.1cm);
    \draw[thick, fill=white] (12, 0.75) circle (0.1cm);
    \draw[thick, fill=white] (7, 1.75) circle (0.1cm);
    \draw[thick, fill=white] (8, 1.75) circle (0.1cm);
    \draw[thick, fill=white] (9, 1.75) circle (0.1cm);
    \draw[thick, fill=white] (10, 1.75) circle (0.1cm);
    \draw[thick, fill=white] (11, 1.75) circle (0.1cm);
    \draw[thick, fill=white] (12, 1.75) circle (0.1cm);
    
    \end{tikzpicture}
    }
    \caption{Two diagrams $b_1$ and $b_1$ (left), concatenated to produce their product (right).}
    \label{fig:4:multiplicationofdiagramsexample}
    
\end{figure}
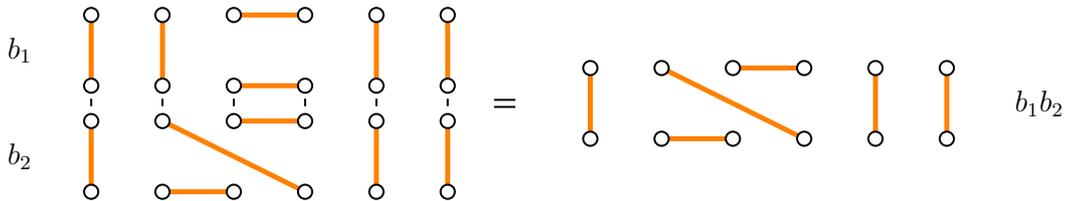

We call the set of all diagrams $B_n$. We call the set of diagrams with exactly $k$ bars $B_n\langle k \rangle$, and the set of diagrams with at least $k$ bars $B_n^k$. Notice that $B_n\langle 0 \rangle$ is exactly the symmetric group $S_n$, and the concatenation multiplication exactly reduced to the multiplication in $S_n$. So $\mathbb{C}S_n$ is a subalgebra of $\mathbb{B}_{n,1}$.

We write $\mathrm{id}$ for the identity in $S_n$ - its diagram has all its edges vertical. We denote the transposition $S_n$ swapping $i$ and $j$ by $(ij)$, and we write $(\overline{ij})$ for the diagram with $i^+$ connected to $j^+$, and $i^-$ connected to $j^-$, and all other edges vertical. See Figure \ref{fig:5:idandtranspositions}.

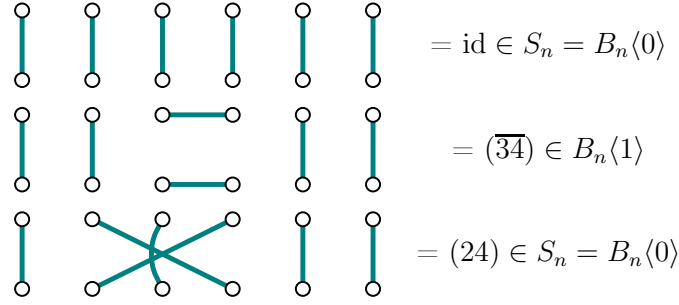
\begin{figure}[h]
    \centering
    \resizebox{0.6\textwidth}{!}{%
    \begin{tikzpicture}[scale=1]
    \draw[font=\large] (7.5, 0.5) node {= $(24) \in S_n = B_n\langle0\rangle$};
    \draw[-, rounded corners=10pt, line width=2pt, teal](1, 1) -- (3, 0);
    \draw[-, rounded corners=10pt, line width=2pt, teal](0, 1) -- (0,0);
    \draw[-, rounded corners=10pt, line width=2pt, teal](2,0) --  (1.75, 0.5) -- (2,1);
    \draw[-, rounded corners=10pt, line width=2pt, teal](3, 1) -- (1, 0);
    \draw[-, rounded corners=10pt, line width=2pt, teal](4,0) -- (4, 1);
    \draw[-, rounded corners=10pt, line width=2pt, teal](5,0) -- (5, 1);
    \draw[thick, fill=white] (0, 1) circle (0.1cm);
    \draw[thick, fill=white] (1, 1) circle (0.1cm);
    \draw[thick, fill=white] (2, 1) circle (0.1cm);
    \draw[thick, fill=white] (3, 1) circle (0.1cm);
    \draw[thick, fill=white] (4, 1) circle (0.1cm);
    \draw[thick, fill=white] (5, 1) circle (0.1cm);
    \draw[thick, fill=white] (0, 0) circle (0.1cm);
    \draw[thick, fill=white] (1, 0) circle (0.1cm);
    \draw[thick, fill=white] (2, 0) circle (0.1cm);
    \draw[thick, fill=white] (3, 0) circle (0.1cm);
    \draw[thick, fill=white] (4, 0) circle (0.1cm);
    \draw[thick, fill=white] (5, 0) circle (0.1cm);
    
    \draw[font=\large] (7.5, 2) node {= $(\overline{34}) \in B_n\langle1\rangle$};
    \draw[-, rounded corners=10pt, line width=2pt, teal](2, 2.5) -- (3, 2.5);
    \draw[-, rounded corners=10pt, line width=2pt, teal](0, 2.5) -- (0, 1.5);
    \draw[-, rounded corners=10pt, line width=2pt, teal](1,2.5) -- (1, 1.5);
    \draw[-, rounded corners=10pt, line width=2pt, teal](2, 1.5) -- (3, 1.5);
    \draw[-, rounded corners=10pt, line width=2pt, teal](5,2.5) -- (5, 1.5);
    \draw[-, rounded corners=10pt, line width=2pt, teal](4,2.5) -- (4, 1.5);
    \draw[thick, fill=white] (0, 1.5) circle (0.1cm);
    \draw[thick, fill=white] (1, 1.5) circle (0.1cm);
    \draw[thick, fill=white] (2, 1.5) circle (0.1cm);
    \draw[thick, fill=white] (3, 1.5) circle (0.1cm);
    \draw[thick, fill=white] (4, 1.5) circle (0.1cm);
    \draw[thick, fill=white] (5, 1.5) circle (0.1cm);
    \draw[thick, fill=white] (0, 2.5) circle (0.1cm);
    \draw[thick, fill=white] (1, 2.5) circle (0.1cm);
    \draw[thick, fill=white] (2, 2.5) circle (0.1cm);
    \draw[thick, fill=white] (3, 2.5) circle (0.1cm);
    \draw[thick, fill=white] (4, 2.5) circle (0.1cm);
    \draw[thick, fill=white] (5, 2.5) circle (0.1cm);
    
    \draw[font=\large] (7.5, 3.5) node {= $\mathrm{id} \in S_n =  B_n\langle0\rangle$};
    \draw[-, rounded corners=10pt, line width=2pt, teal](3,3) -- (3, 4);
    \draw[-, rounded corners=10pt, line width=2pt, teal](0,3) -- (0, 4);
    \draw[-, rounded corners=10pt, line width=2pt, teal](1,3) -- (1, 4);
    \draw[-, rounded corners=10pt, line width=2pt, teal](2,3) -- (2, 4);
    \draw[-, rounded corners=10pt, line width=2pt, teal](5,3) -- (5, 4);
    \draw[-, rounded corners=10pt, line width=2pt, teal](4,3) -- (4, 4);
    \draw[thick, fill=white] (0, 3) circle (0.1cm);
    \draw[thick, fill=white] (1, 3) circle (0.1cm);
    \draw[thick, fill=white] (2, 3) circle (0.1cm);
    \draw[thick, fill=white] (3, 3) circle (0.1cm);
    \draw[thick, fill=white] (4, 3) circle (0.1cm);
    \draw[thick, fill=white] (5, 3) circle (0.1cm);
    \draw[thick, fill=white] (0, 4) circle (0.1cm);
    \draw[thick, fill=white] (1, 4) circle (0.1cm);
    \draw[thick, fill=white] (2, 4) circle (0.1cm);
    \draw[thick, fill=white] (3, 4) circle (0.1cm);
    \draw[thick, fill=white] (4, 4) circle (0.1cm);
    \draw[thick, fill=white] (5, 4) circle (0.1cm);
    
    \end{tikzpicture}
    }
    \caption{The identity element, the element $(\overline{34}) \in B_n\langle1\rangle$, and the transposition $(24) \in S_n = B_N\langle0\rangle$.}
    \label{fig:5:idandtranspositions}
    
\end{figure}

Finally, we remark that if $b$ has $k$ bars, and $c$ is any diagram in $B_n$, then $bc$ must have at least $k$ bars:
\begin{equation}\label{rmk:multdoesntremovebars}
    b \in B_{n}\langle k \rangle \ \Rightarrow \ bc \in B_n^k
\end{equation}\\

Let us now see how the Brauer algebra can be used to describe the models. Let $n$ be even from hereon in. Observe that given a configuration $\sigma_i$ of mirrors on a street $s_i$ on the $n$-cylinder, the paths through the street form a diagram $b(\sigma_i) \in B_n$. See Figure \ref{fig:6:pathstodiagrams} for an illustration. Moreover, on any section of the cylinder, say, from street $s_i$ to $s_j$, given a configuration of mirrors $\sigma_{i \to j}$ the paths through those streets form a diagram $b(\sigma_{i \to j})$. Crucially, we see that $b(\sigma_{i \to j}) = b(\sigma_i) \cdots b(\sigma_j)$, where the multiplication on the right hand side is in the Brauer algebra. See Figure \ref{fig:7:concatenatingstreets}.

\begin{figure}[h]
    
\centering
\resizebox{1.0\textwidth}{!}{%
\begin{tikzpicture}[scale=1.4, every node/.style={transform shape}]
    
    \draw[-{Latex[length=0.4cm]}] (1,-0.9) -- (1,2.9);
    \draw[-{Latex[length=0.4cm]}] (3,-0.9) -- (3,2.9);
    \draw[-{Latex[length=0.4cm]}] (5,-0.9) -- (5,2.9);
    \draw[-{Latex[length=0.4cm]}] (0,2.9) -- (0,-0.9);
    \draw[-{Latex[length=0.4cm]}] (2,2.9) -- (2,-0.9);
    \draw[-{Latex[length=0.4cm]}] (4,2.9) -- (4,-0.9);
    \draw[-{Latex[length=0.4cm]}] (5.9,2) -- (-0.9,2);
    \draw[-{Latex[length=0.4cm]}] (5.9,0) -- (-0.9,0);
    \draw[-{Latex[length=0.4cm]}] (-0.9,1) -- (5.9,1);
    
    \URobstacle[teal](3:1);
    \ULobstacle[teal](0:1);
    \ULobstacle[teal](2:1);
    
    \draw[-{Latex[length=0.4cm]}, rounded corners=9pt, line width=3pt, orange](-0.9, 1) -- (0,1) -- (0,0.5);
    \draw[-{Latex[length=0.4cm]}, rounded corners=9pt, line width=3pt, orange](1, 0.5) -- (1, 1.5);
    \draw[-{Latex[length=0.4cm]}, rounded corners=9pt, line width=3pt, orange](0, 1.5) -- (0, 1) -- (2,1) -- (2,0.5);
    \draw[-{Latex[length=0.4cm]}, rounded corners=9pt, line width=3pt, orange](2, 1.5) -- (2, 1) -- (3,1) -- (3,1.5);
    \draw[-{Latex[length=0.4cm]}, rounded corners=9pt, line width=3pt, orange](3, 0.5) -- (3, 1) -- (5.9,1);
    \draw[-{Latex[length=0.4cm]}, rounded corners=9pt, line width=3pt, orange](4, 1.5) -- (4, 0.5);
    \draw[-{Latex[length=0.4cm]}, rounded corners=9pt, line width=3pt, orange](5, 0.5) -- (5, 1.5);
    
    \draw[thick, fill=white] (0, 0.5) circle (0.1cm);
    \draw[thick, fill=white] (1, 0.5) circle (0.1cm);
    \draw[thick, fill=white] (2, 0.5) circle (0.1cm);
    \draw[thick, fill=white] (3, 0.5) circle (0.1cm);
    \draw[thick, fill=white] (4, 0.5) circle (0.1cm);
    \draw[thick, fill=white] (5, 0.5) circle (0.1cm);
    \draw[thick, fill=white] (0, 1.5) circle (0.1cm);
    \draw[thick, fill=white] (1, 1.5) circle (0.1cm);
    \draw[thick, fill=white] (2, 1.5) circle (0.1cm);
    \draw[thick, fill=white] (3, 1.5) circle (0.1cm);
    \draw[thick, fill=white] (4, 1.5) circle (0.1cm);
    \draw[thick, fill=white] (5, 1.5) circle (0.1cm);
    
    \end{tikzpicture}
    \quad
    \begin{tikzpicture}[scale=1.4, every node/.style={transform shape}]
    
    \draw[-{Latex[length=0.4cm]},white] (1,-0.9) -- (1,2.9);
    \draw[-{Latex[length=0.4cm]},white] (3,-0.9) -- (3,2.9);
    \draw[-{Latex[length=0.4cm]},white] (5,-0.9) -- (5,2.9);
    \draw[-{Latex[length=0.4cm]},white] (0,2.9) -- (0,-0.9);
    \draw[-{Latex[length=0.4cm]},white] (2,2.9) -- (2,-0.9);
    \draw[-{Latex[length=0.4cm]},white] (4,2.9) -- (4,-0.9);
    \draw[-{Latex[length=0.4cm]},white] (5.9,2) -- (-0.9,2);
    \draw[-{Latex[length=0.4cm]},white] (5.9,0) -- (-0.9,0);
    \draw[-{Latex[length=0.4cm]},white] (-0.9,1) -- (5.9,1);
    
    \draw[-, rounded corners=40pt, line width=3pt, orange](3, 0.5) -- (1.5,1) -- (0,0.5);
    \draw[-, rounded corners=9pt, line width=3pt, orange](1, 0.5) -- (1, 1.5);
    \draw[-, rounded corners=9pt, line width=3pt, orange](0, 1.5) -- (2,0.5);
    \draw[-, rounded corners=15pt, line width=3pt, orange](2, 1.5) -- (2.5, 1.25) -- (3,1.5);
    \draw[-, rounded corners=9pt, line width=3pt, orange](4, 1.5) -- (4, 0.5);
    \draw[-, rounded corners=9pt, line width=3pt, orange](5, 0.5) -- (5, 1.5);
    
    \draw[thick, fill=white] (0, 0.5) circle (0.1cm);
    \draw[thick, fill=white] (1, 0.5) circle (0.1cm);
    \draw[thick, fill=white] (2, 0.5) circle (0.1cm);
    \draw[thick, fill=white] (3, 0.5) circle (0.1cm);
    \draw[thick, fill=white] (4, 0.5) circle (0.1cm);
    \draw[thick, fill=white] (5, 0.5) circle (0.1cm);
    \draw[thick, fill=white] (0, 1.5) circle (0.1cm);
    \draw[thick, fill=white] (1, 1.5) circle (0.1cm);
    \draw[thick, fill=white] (2, 1.5) circle (0.1cm);
    \draw[thick, fill=white] (3, 1.5) circle (0.1cm);
    \draw[thick, fill=white] (4, 1.5) circle (0.1cm);
    \draw[thick, fill=white] (5, 1.5) circle (0.1cm);
    
    \end{tikzpicture}
    
}%
    
    \caption{An example of a configuration of mirrors $\sigma_i$ on street $s_i$ in the Manhattan model (left), and the resulting diagram $b(\sigma_i)$ (right).}
    \label{fig:6:pathstodiagrams}
\end{figure}
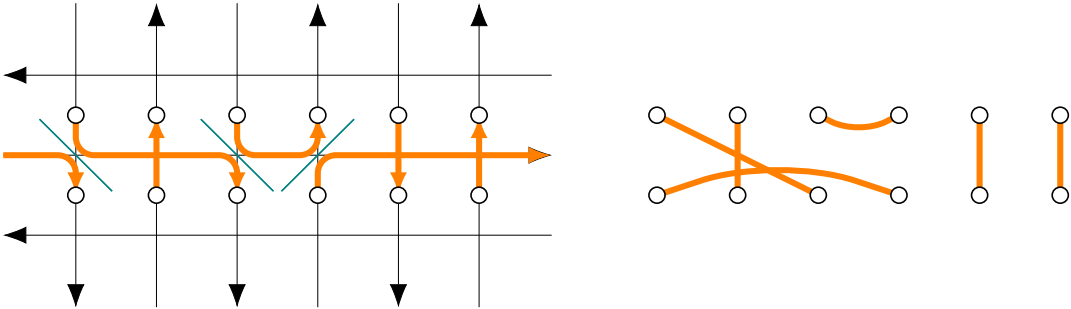

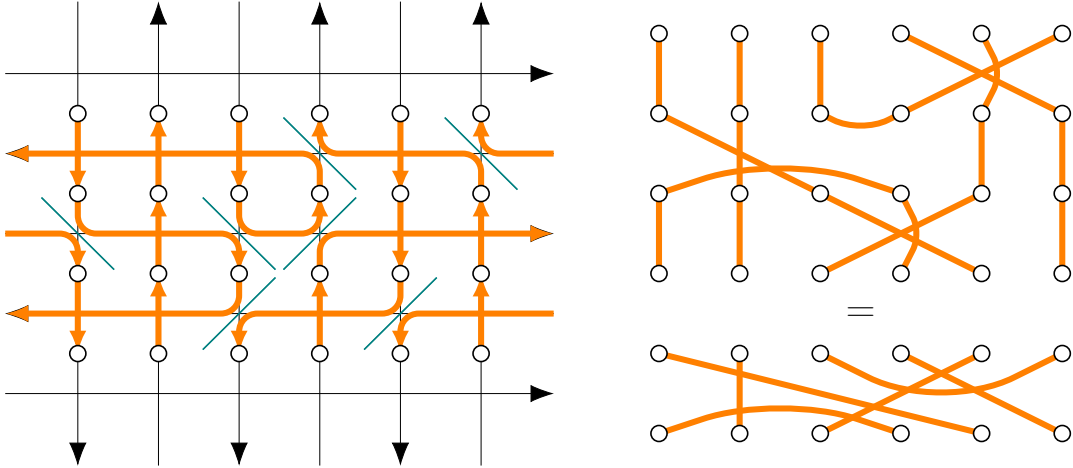
\begin{figure}[h]
    
\centering
\resizebox{1.0\textwidth}{!}{%
\begin{tikzpicture}[scale=1.4, every node/.style={transform shape}]
    
    \draw[-{Latex[length=0.4cm]}] (1,-0.9) -- (1,4.9);
    \draw[-{Latex[length=0.4cm]}] (3,-0.9) -- (3,4.9);
    \draw[-{Latex[length=0.4cm]}] (5,-0.9) -- (5,4.9);
    \draw[-{Latex[length=0.4cm]}] (0,4.9) -- (0,-0.9);
    \draw[-{Latex[length=0.4cm]}] (2,4.9) -- (2,-0.9);
    \draw[-{Latex[length=0.4cm]}] (4,4.9) -- (4,-0.9);
    \draw[-{Latex[length=0.4cm]}] (5.9,1) -- (-0.9,1);
    \draw[-{Latex[length=0.4cm]}] (5.9,3) -- (-0.9,3);
    \draw[-{Latex[length=0.4cm]}] (-0.9,0) -- (5.9,0);
    \draw[-{Latex[length=0.4cm]}] (-0.9,2) -- (5.9,2);
    \draw[-{Latex[length=0.4cm]}] (-0.9,4) -- (5.9,4);
    
    \URobstacle[teal](2:1);
    \URobstacle[teal](4:1);
    \URobstacle[teal](3:2);
    \ULobstacle[teal](0:2);
    \ULobstacle[teal](2:2);
    \ULobstacle[teal](3:3);
    \ULobstacle[teal](5:3);
    
    \draw[-{Latex[length=0.4cm]}, rounded corners=9pt, line width=3pt, orange](0, 3.5) -- (0, 2.5);
    \draw[-{Latex[length=0.4cm]}, rounded corners=9pt, line width=3pt, orange](1, 2.5) -- (1, 3.5);
    \draw[-{Latex[length=0.4cm]}, rounded corners=9pt, line width=3pt, orange](3, 2.5) -- (3, 3) -- (-0.9, 3);
    \draw[-{Latex[length=0.4cm]}, rounded corners=9pt, line width=3pt, orange](5, 2.5) -- (5, 3) -- (3, 3) -- (3, 3.5);
    \draw[-{Latex[length=0.4cm]}, rounded corners=9pt, line width=3pt, orange](4, 3.5) -- (4, 2.5);
    \draw[-{Latex[length=0.4cm]}, rounded corners=9pt, line width=3pt, orange](5.9, 3) -- (5, 3) -- (5, 3.5);
    \draw[-{Latex[length=0.4cm]}, rounded corners=9pt, line width=3pt, orange](2, 3.5) -- (2, 2.5);
    \draw[-{Latex[length=0.4cm]}, rounded corners=9pt, line width=3pt, orange](-0.9, 2) -- (0,2) -- (0,1.5);
    \draw[-{Latex[length=0.4cm]}, rounded corners=9pt, line width=3pt, orange](1, 1.5) -- (1, 2.5);
    \draw[-{Latex[length=0.4cm]}, rounded corners=9pt, line width=3pt, orange](0, 2.5) -- (0, 2) -- (2,2) -- (2,1.5);
    \draw[-{Latex[length=0.4cm]}, rounded corners=9pt, line width=3pt, orange](2, 2.5) -- (2, 2) -- (3,2) -- (3,2.5);
    \draw[-{Latex[length=0.4cm]}, rounded corners=9pt, line width=3pt, orange](3, 1.5) -- (3,2) -- (5.9,2);
    \draw[-{Latex[length=0.4cm]}, rounded corners=9pt, line width=3pt, orange](4, 2.5) -- (4, 1.5);
    \draw[-{Latex[length=0.4cm]}, rounded corners=9pt, line width=3pt, orange](5, 1.5) -- (5, 2.5);
    \draw[-{Latex[length=0.4cm]}, rounded corners=9pt, line width=3pt, orange](0, 1.5) -- (0, 0.5);
    \draw[-{Latex[length=0.4cm]}, rounded corners=9pt, line width=3pt, orange](1, 0.5) -- (1, 1.5);
    \draw[-{Latex[length=0.4cm]}, rounded corners=9pt, line width=3pt, orange](2, 1.5) -- (2, 1) -- (-0.9, 1);
    \draw[-{Latex[length=0.4cm]}, rounded corners=9pt, line width=3pt, orange](4, 1.5) -- (4, 1) -- (2, 1) -- (2, 0.5);
    \draw[-{Latex[length=0.4cm]}, rounded corners=9pt, line width=3pt, orange](3, 0.5) -- (3, 1.5);
    \draw[-{Latex[length=0.4cm]}, rounded corners=9pt, line width=3pt, orange](5.9, 1) -- (4, 1) -- (4, 0.5);
    \draw[-{Latex[length=0.4cm]}, rounded corners=9pt, line width=3pt, orange](5, 0.5) -- (5, 1.5);
    
    \draw[thick, fill=white] (0, 0.5) circle (0.1cm);
    \draw[thick, fill=white] (1, 0.5) circle (0.1cm);
    \draw[thick, fill=white] (2, 0.5) circle (0.1cm);
    \draw[thick, fill=white] (3, 0.5) circle (0.1cm);
    \draw[thick, fill=white] (4, 0.5) circle (0.1cm);
    \draw[thick, fill=white] (5, 0.5) circle (0.1cm);
    \draw[thick, fill=white] (0, 1.5) circle (0.1cm);
    \draw[thick, fill=white] (1, 1.5) circle (0.1cm);
    \draw[thick, fill=white] (2, 1.5) circle (0.1cm);
    \draw[thick, fill=white] (3, 1.5) circle (0.1cm);
    \draw[thick, fill=white] (4, 1.5) circle (0.1cm);
    \draw[thick, fill=white] (5, 1.5) circle (0.1cm);
    \draw[thick, fill=white] (0, 2.5) circle (0.1cm);
    \draw[thick, fill=white] (1, 2.5) circle (0.1cm);
    \draw[thick, fill=white] (2, 2.5) circle (0.1cm);
    \draw[thick, fill=white] (3, 2.5) circle (0.1cm);
    \draw[thick, fill=white] (4, 2.5) circle (0.1cm);
    \draw[thick, fill=white] (5, 2.5) circle (0.1cm);
    \draw[thick, fill=white] (0, 3.5) circle (0.1cm);
    \draw[thick, fill=white] (1, 3.5) circle (0.1cm);
    \draw[thick, fill=white] (2, 3.5) circle (0.1cm);
    \draw[thick, fill=white] (3, 3.5) circle (0.1cm);
    \draw[thick, fill=white] (4, 3.5) circle (0.1cm);
    \draw[thick, fill=white] (5, 3.5) circle (0.1cm);
    
    \end{tikzpicture}
    \quad
    \begin{tikzpicture}[scale=1.4, every node/.style={transform shape}]
    \draw[font=\large] (2.5,0) node {\textbf{=}};
    
    \draw[-{Latex[length=0.4cm]},white] (1,-1.9) -- (1,4.9);
    \draw[-{Latex[length=0.4cm]},white] (3,-1.9) -- (3,4.9);
    \draw[-{Latex[length=0.4cm]},white] (5,-1.9) -- (5,4.9);
    \draw[-{Latex[length=0.4cm]},white] (0,4.9) -- (0,-1.9);
    \draw[-{Latex[length=0.4cm]},white] (2,4.9) -- (2,-1.9);
    \draw[-{Latex[length=0.4cm]},white] (4,4.9) -- (4,-1.9);
    \draw[-{Latex[length=0.4cm]},white] (5.9,1) -- (-0.9,1);
    \draw[-{Latex[length=0.4cm]},white] (5.9,3) -- (-0.9,3);
    \draw[-{Latex[length=0.4cm]},white] (-0.9,0) -- (5.9,0);
    \draw[-{Latex[length=0.4cm]},white] (-0.9,2) -- (5.9,2);
    \draw[-{Latex[length=0.4cm]},white] (-0.9,4) -- (5.9,4);
    
    \draw[-, rounded corners=40pt, line width=3pt, orange](0, 2.5) -- (0,3.5);
    \draw[-, rounded corners=9pt, line width=3pt, orange](1, 3.5) -- (1, 2.5);
    \draw[-, rounded corners=9pt, line width=3pt, orange](2, 2.5) -- (2,3.5);
    \draw[-, rounded corners=15pt, line width=3pt, orange](3, 2.5) -- (5,3.5);
    \draw[-, rounded corners=9pt, line width=3pt, orange](3, 3.5) -- (5, 2.5);
    \draw[-, rounded corners=9pt, line width=3pt, orange](4, 2.5) -- (4.25, 3) -- (4, 3.5);
    \draw[-, rounded corners=40pt, line width=3pt, orange](3, 1.5) -- (1.5,2) -- (0,1.5);
    \draw[-, rounded corners=9pt, line width=3pt, orange](1, 1.5) -- (1, 2.5);
    \draw[-, rounded corners=9pt, line width=3pt, orange](0, 2.5) -- (2,1.5);
    \draw[-, rounded corners=15pt, line width=3pt, orange](2, 2.5) -- (2.5, 2.25) -- (3,2.5);
    \draw[-, rounded corners=9pt, line width=3pt, orange](4, 2.5) -- (4, 1.5);
    \draw[-, rounded corners=9pt, line width=3pt, orange](5, 1.5) -- (5, 2.5);
    \draw[-, rounded corners=40pt, line width=3pt, orange](0, 1.5) -- (0,0.5);
    \draw[-, rounded corners=9pt, line width=3pt, orange](1, 1.5) -- (1, 0.5);
    \draw[-, rounded corners=9pt, line width=3pt, orange](2, 1.5) -- (4, 0.5);
    \draw[-, rounded corners=15pt, line width=3pt, orange](2, 0.5) -- (4, 1.5);
    \draw[-, rounded corners=9pt, line width=3pt, orange](3, 0.5) -- (3.25, 1) -- (3, 1.5);
    \draw[-, rounded corners=9pt, line width=3pt, orange](5, 0.5) -- (5, 1.5);
    \draw[-, rounded corners=40pt, line width=3pt, orange](0, -0.5) -- (4,-1.5);
    \draw[-, rounded corners=40pt, line width=3pt, orange](0, -1.5) -- (1.5, -1) -- (3, -1.5);
    \draw[-, rounded corners=9pt, line width=3pt, orange](2, -1.5) -- (4, -0.5);
    \draw[-, rounded corners=40pt, line width=3pt, orange](2, -0.5) -- (3.5, -1.25) -- (5, -0.5);
    \draw[-, rounded corners=9pt, line width=3pt, orange](3, -0.5) -- (5, -1.5);
    \draw[-, rounded corners=9pt, line width=3pt, orange](1, -0.5) -- (1, -1.5);
    
    %
    \draw[thick, fill=white] (0, -1.5) circle (0.1cm);
    \draw[thick, fill=white] (1, -1.5) circle (0.1cm);
    \draw[thick, fill=white] (2, -1.5) circle (0.1cm);
    \draw[thick, fill=white] (3, -1.5) circle (0.1cm);
    \draw[thick, fill=white] (4, -1.5) circle (0.1cm);
    \draw[thick, fill=white] (5, -1.5) circle (0.1cm);
    \draw[thick, fill=white] (0, -0.5) circle (0.1cm);
    \draw[thick, fill=white] (1, -0.5) circle (0.1cm);
    \draw[thick, fill=white] (2, -0.5) circle (0.1cm);
    \draw[thick, fill=white] (3, -0.5) circle (0.1cm);
    \draw[thick, fill=white] (4, -0.5) circle (0.1cm);
    \draw[thick, fill=white] (5, -0.5) circle (0.1cm);
    \draw[thick, fill=white] (0, 0.5) circle (0.1cm);
    \draw[thick, fill=white] (1, 0.5) circle (0.1cm);
    \draw[thick, fill=white] (2, 0.5) circle (0.1cm);
    \draw[thick, fill=white] (3, 0.5) circle (0.1cm);
    \draw[thick, fill=white] (4, 0.5) circle (0.1cm);
    \draw[thick, fill=white] (5, 0.5) circle (0.1cm);
    \draw[thick, fill=white] (0, 1.5) circle (0.1cm);
    \draw[thick, fill=white] (1, 1.5) circle (0.1cm);
    \draw[thick, fill=white] (2, 1.5) circle (0.1cm);
    \draw[thick, fill=white] (3, 1.5) circle (0.1cm);
    \draw[thick, fill=white] (4, 1.5) circle (0.1cm);
    \draw[thick, fill=white] (5, 1.5) circle (0.1cm);
    \draw[thick, fill=white] (0, 2.5) circle (0.1cm);
    \draw[thick, fill=white] (1, 2.5) circle (0.1cm);
    \draw[thick, fill=white] (2, 2.5) circle (0.1cm);
    \draw[thick, fill=white] (3, 2.5) circle (0.1cm);
    \draw[thick, fill=white] (4, 2.5) circle (0.1cm);
    \draw[thick, fill=white] (5, 2.5) circle (0.1cm);
    \draw[thick, fill=white] (0, 3.5) circle (0.1cm);
    \draw[thick, fill=white] (1, 3.5) circle (0.1cm);
    \draw[thick, fill=white] (2, 3.5) circle (0.1cm);
    \draw[thick, fill=white] (3, 3.5) circle (0.1cm);
    \draw[thick, fill=white] (4, 3.5) circle (0.1cm);
    \draw[thick, fill=white] (5, 3.5) circle (0.1cm);
    
    \end{tikzpicture}
    
}%
    
    \caption{The paths through three consecutive streets in the Manhattan model (left), the three resulting diagrams (upper right), and their product (lower right), which gives the paths through the union of the three streets.}
    \label{fig:7:concatenatingstreets}
\end{figure}

Each $\sigma_i$ is randomly distributed, so $b(\sigma_i)$ is a random diagram in $B_n$. We can think of the distribution of this random diagram as an element $Z^{(i)}$ of the algebra, as:
\begin{equation*}
    Z^{(i)} = \sum_{g \in B_n} \mathbb{P}[b(\sigma_i) = g] \cdot g
\end{equation*}
The following lemma lets us describe the paths through any number of consecutive streets. Note that it does not use the specific distributions of the random diagrams given by different streets, it only uses that they are independent.
\begin{lem}
    The distribution of the random diagram (as an element of the Brauer algebra) produced by the paths through streets $s_i, \dots s_j$ is given by:
    \begin{equation*}
        Z^{(i)} \cdots Z^{(j)}
    \end{equation*}
\end{lem}
\begin{proof}
We see that
\begin{equation*}
    \begin{split}
        \sum_{g \in B_n} \mathbb{P}[b(\sigma_{i \to j}) = g] \cdot g 
        &=
        \sum_{g \in B_n} \mathbb{P}[b(\sigma_{i}) \cdots b(\sigma_j) = g] \cdot g \\
        &=
        \sum_{g \in B_n} \sum_{g_i \cdots g_j = g} \mathbb{P}[b(\sigma_i) = g_i] \cdots \mathbb{P}[b(\sigma_j) = g_j] \cdot g_i \cdots g_j \\
        &=
        Z^{(i)} \cdots Z^{(j)}
    \end{split}
\end{equation*}
where we use that the configurations on each street are independent. 
\end{proof}
We are interested in the highest (or most northerly) street reached by the paths starting at the first street $s_1$. One more than this is the first street which has no path connecting it to $s_1$. Using the notation above, this is the smallest $i$ such that 
\begin{equation*}
    Z^{(1)} \cdots Z^{(i)} \in B_n \langle \frac{n}{2} \rangle.
\end{equation*}

Now, in the Mirror model, the distribution of mirrors is iid on each street. Let $Z_{mir}$ be the distribution of the random diagram (as an element of the Brauer algebra) produced by the paths through this random configuration on one street. Let $V^{mir}_{k}$ be the smallest $i$ such that $(Z_{mir})^i \in B_n\langle k \rangle$. We're primarily interested in $V^{mir}_{\frac{n}{2}}$. \\

The Manhattan model is almost identical in this regard, with two differences. The first is that the distribution of mirrors on a street is dependent on whether the street is eastbound or westbound. We can let $Z_{(mat, E)}, Z_{(mat, W)}$ be the random diagrams that arise in these cases, respectively. 

Secondly, the diagrams that arise in the Manhattan case actually live in a subalgebra of $\mathbb{B}_{n,1}$. Note that each vertical column of the cylinder $\mathbb{Z} \times \{ i\}$ is southbound for $i$ odd, northbound for $i$ even. This means that on a chosen street, the vertices $i^+$ for $i$ odd, and $i^-$ for $i$ even, can be thought of as "entrypoints" to the street. Similarly, each $j^+$ for $j$ even, $j^-$ for $j$ odd can be thought of as "exitpoints" to the street. In particular, in the diagram which results from the street, \textit{exitpoints must be connected to entrypoints}. This condition can also be thought of as: a NS path must connect vertices of the same parity, and a bar must connect vertices of different parity. See Figure \ref{fig:8:exit&entrypoints}.

\begin{figure}[h]
    
\centering
\resizebox{0.5\textwidth}{!}{%
\begin{tikzpicture}[scale=1.4, every node/.style={transform shape}]
    
    \draw[-{Latex[length=0.4cm]}] (1,-0.9) -- (1,2.9);
    \draw[-{Latex[length=0.4cm]}] (3,-0.9) -- (3,2.9);
    \draw[-{Latex[length=0.4cm]}] (5,-0.9) -- (5,2.9);
    \draw[-{Latex[length=0.4cm]}] (0,2.9) -- (0,-0.9);
    \draw[-{Latex[length=0.4cm]}] (2,2.9) -- (2,-0.9);
    \draw[-{Latex[length=0.4cm]}] (4,2.9) -- (4,-0.9);
    \draw[-{Latex[length=0.4cm]}] (5.9,2) -- (-0.9,2);
    \draw[-{Latex[length=0.4cm]}] (5.9,0) -- (-0.9,0);
    \draw[-{Latex[length=0.4cm]}] (-0.9,1) -- (5.9,1);
    
    \URobstacle[teal](3:1);
    \ULobstacle[teal](0:1);
    \ULobstacle[teal](2:1);
    
    \draw[-{Latex[length=0.4cm]}, rounded corners=9pt, line width=3pt, orange](-0.9, 1) -- (0,1) -- (0,0.5);
    \draw[-{Latex[length=0.4cm]}, rounded corners=9pt, line width=3pt, orange](1, 0.5) -- (1, 1.5);
    \draw[-{Latex[length=0.4cm]}, rounded corners=9pt, line width=3pt, orange](0, 1.5) -- (0, 1) -- (2,1) -- (2,0.5);
    \draw[-{Latex[length=0.4cm]}, rounded corners=9pt, line width=3pt, orange](2, 1.5) -- (2, 1) -- (3,1) -- (3,1.5);
    \draw[-{Latex[length=0.4cm]}, rounded corners=9pt, line width=3pt, orange](3, 0.5) -- (3, 1) -- (5.9,1);
    \draw[-{Latex[length=0.4cm]}, rounded corners=9pt, line width=3pt, orange](4, 1.5) -- (4, 0.5);
    \draw[-{Latex[length=0.4cm]}, rounded corners=9pt, line width=3pt, orange](5, 0.5) -- (5, 1.5);
    
    \draw[thick, fill=teal] (0, 0.5) circle (0.1cm);
    \draw[thick, fill=yellow] (1, 0.5) circle (0.1cm);
    \draw[thick, fill=teal] (2, 0.5) circle (0.1cm);
    \draw[thick, fill=yellow] (3, 0.5) circle (0.1cm);
    \draw[thick, fill=teal] (4, 0.5) circle (0.1cm);
    \draw[thick, fill=yellow] (5, 0.5) circle (0.1cm);
    \draw[thick, fill=yellow] (0, 1.5) circle (0.1cm);
    \draw[thick, fill=teal] (1, 1.5) circle (0.1cm);
    \draw[thick, fill=yellow] (2, 1.5) circle (0.1cm);
    \draw[thick, fill=teal] (3, 1.5) circle (0.1cm);
    \draw[thick, fill=yellow] (4, 1.5) circle (0.1cm);
    \draw[thick, fill=teal] (5, 1.5) circle (0.1cm);
    
    \end{tikzpicture}
    
}%
    
    \caption{An example of paths through a street in the Manhattan model, with entrypoints coloured in yellow, and exitpoints in blue.}
    \label{fig:8:exit&entrypoints}
\end{figure}

Let $M_n$ be the set of diagrams which satisfy the requirement that exitpoints are only connected to entrypoints, and let $\mathbb{M}_{n,1}$ be the (formal) complex span of $M_n$. This space $\mathbb{M}_{n,1}$ is a subalgebra of $\mathbb{B}_{n,1}$, indeed it is a special case of the Walled Brauer algebra. See \cite{coxwalled08}. Similar to the full algebra, let $M_n\langle k \rangle$ be the set of diagrams in $M_n$ with $k$ bars, and let $M_n^k$ be those with at least $k$ bars. 

Let's assume that the first street, $s_1$, is eastbound. Now let $V^{mat}_k$ be the random variable given by the smallest $i$ such that \begin{equation*}
    \underbrace{Z_{(mat, E)} Z_{(mat, W)} Z_{(mat, E)} \cdots}_{\text{i terms}} 
    =
    \begin{cases}
    \left( Z_{(mat, E)} Z_{(mat, W)} \right)^{\frac{i}{2}} 
    & i \ \mathrm{even} \\
    \left( Z_{(mat, E)} Z_{(mat, W)} \right)^{\frac{i-1}{2}}Z_{(mat, E)}
    & i \ \mathrm{odd}
    \end{cases} \ 
    \in M_n \langle k \rangle,
\end{equation*}
where the equality is included for clarity. In section \ref{section:results} we will write $Z_{mat}^{[i]}$ as shorthand for the product of $i$ terms on the left hand side. We are primarily interested in $V^{mat}_{\frac{n}{2}}$. \\

Now recall that our method is to condition on there being at most two obstacles per street. Let $U_{\le 2}$ be the event that there are at most two mirrors on a street (this is independent of which model we're considering). Let $X_{mir}$, $X_{mat}$ be the distributions $Z_{mir}$, $Z_{(mat, E/W)}$, conditioned on $U_{\le 2}$, respectively. (In the Manhattan case, it actually doesn't matter whether the street is eastbound or westbound). We write $X_{mir}, X_{mat}$ as elements of the Brauer algebra:
\begin{equation*}
    X_{mir} = \frac{(1-p)^{n-2}}{\mathbb{P}[U_{\le 2}]} 
        \left[
        \left( np(1-p) + (1-p)^2 \right)\cdot \mathrm{id} 
        +
        \frac{p^2}{2} \left( \sum_{1 \le i < j \le n} (ij) + (\overline{ij}) \right)
        \right]
\end{equation*}
and very similarly:
\begin{equation*}
    X_{mat} = \frac{(1-p)^{n-2}}{\mathbb{P}[U_{\le 2}]} 
        \left[
        \left( np(1-p) + (1-p)^2 \right)\cdot \mathrm{id} 
        +
        p^2 \left( \sum_{j-i \ \mathrm{even}} (ij) 
        + \sum_{j-i \ \mathrm{odd}} (\overline{ij}) \right)
        \right]
\end{equation*}
where we recall that the diagrams $(ij)$, $(\overline{ij})$, and $\mathrm{id}$ are given in Figure 5 ***. Note that $Z_{mir}$, $Z_{(mat, E/W)}$ can also be explicitly written down as elements of the Brauer algebra (for any $p$), they are just far more unwieldy. 

Similar to above, let $*$ denote $mir$ or $mat$, and define $W^*_k$ to be the random variable given by the smallest $i$ such that $(X_*)^i \in B_n \langle \frac{n}{2} \rangle$. In the next section, we give bounds on how large $W^*_{\frac{n}{2}}$ can be, and then we transfer these bounds to $V^*_{\frac{n}{2}}$.

\section{Results}\label{section:results}

Let us first prove part $b)$ of Theorem \ref{thm:1.1mainthm}. Let $*$ denote $mir$ or $mat$. Let $G$ be a geometric random variable with parameter $p^2$. We first show that $\mathbb{P}[V_{\frac{n}{2}}^* \le x] \le \mathbb{P}[G \le x]$, for all $x \ge 0$. 

Let us write, in the Manhattan case, $Z_{mat}^{[i]}$ as shorthand for the product containing $i$ factors $Z_{(mat, E)} Z_{(mat, W)} Z_{(mat, E)} \cdots$, and for consistency, let us write $Z_{mir}^{[i]} = Z_{mir}^i$. Assume that $Z_*^{[i]} \notin B_n \langle \frac{n}{2}\rangle$, that is, after $i$ streets, there are at least two remaining NS paths. Consider the probability $\mathbb{P}[Z_*^{[i+1]} \in B_n \langle \frac{n}{2} \rangle]$, that after the next street, no NS paths remain. In order for $Z_*^{[i+1]} \in B_n \langle \frac{n}{2} \rangle$ to hold, there certainly must be a mirror on $s_{i+1}$ reflecting each of the remaining NS paths - since there are at least two of these, the probability of this is at most $p^2$. Hence we can say that, given that $Z_*^{[i]} \notin B_n \langle \frac{n}{2}\rangle$,
\begin{equation*}
    \mathbb{P}[Z_*^{[i+1]} \in B_n \langle \frac{n}{2} \rangle] \le p^2.
\end{equation*}
Now we can easily couple the process with one which enters $B_n \langle \frac{n}{2} \rangle$ at each step with probability exactly $p^2$. The time taken for this process to enter $B_n \langle \frac{n}{2} \rangle$ can be described by $G$, and our claim  $\mathbb{P}[V_{\frac{n}{2}}^* \le x] \le \mathbb{P}[G \le x]$ follows. Now for $p \le \frac{1}{2}$,
\begin{equation*}
    \mathbb{P}[V_{\frac{n}{2}}^* \le \alpha p^{-2}] 
    \le 
    \mathbb{P}[G \le \alpha p^{-2}]
    = 
    1- (1-p^2)^{\alpha p^{-2}}
    \le 
    2 \alpha,
\end{equation*}
the last inequality following from both functions taking the value 0 at $\alpha=0$, and the differential of the first function being $(1-p^2)^{\alpha p^{-2}} \log ((1-p^2)^{p^{-2}})$, whose value is less than $2$ at $\alpha = 0$ and decreasing as $\alpha$ increases. This completes the proof of part $b)$.
\\

The rest of this section proves part $a)$ of Theorem \ref{thm:1.1mainthm}. We return to our simplified model, assuming at most two mirrors on each street. Observe that if the random diagram $X_*$ is multiplied with a diagram $b$ which has $k$ bars, the probability that the result has $k+1$ bars is \textit{independent of the chosen diagram b}. This is made precise in the following lemma.

\begin{lem}\label{lem:3:add1bar}
    \begin{enumerate}[label=\alph*)]
        \item Let $b \in B_{n}\langle k \rangle$, a diagram with $k$ bars. Then $bX_{mir} \in B_n\langle k \rangle \cup B_n\langle k+1 \rangle$, and 
            \begin{equation*}
                g_{n,p,k}^{mir} := \mathbb{P}[bX_{mir} \in B_n\langle k+1 \rangle] = \frac{1}{\mathbb{P}[U_{\le 2}]} \frac{p^2}{2} (1-p)^{n-2} {n-2k \choose 2}.
            \end{equation*}
        \item Let $b \in M_{n}\langle k \rangle$, a diagram with $k$ bars. Then $bX_{mat} \in M_n\langle k \rangle \cup M_n\langle k+1 \rangle$, and 
            \begin{equation*}
                g_{n,p,k}^{mat} := \mathbb{P}[bX_{mat} \in M_n\langle k+1 \rangle] = \frac{1}{\mathbb{P}[U_{\le 2}]} p^2 (1-p)^{n-2} (\frac{n}{2}-k)^2.
            \end{equation*}
    \end{enumerate}
    
\end{lem}

\begin{proof}
    Let's do part $a)$ first. Let $b \in B_n\langle k \rangle$. It's clear that $b(ij) \in B_n\langle k \rangle$. Further, $b(\overline{ij}) \in B_n\langle k+1 \rangle$ iff the vertices $i^-$ and $j^-$ in $b$ lie on NS paths. There are ${n-2k \choose 2}$ such pairs. So, 
    \begin{equation*}
        \mathbb{P}[X_{mir}=(\overline{ij}), \ i,j \ \mathrm{on \ NS \ paths \ in} \ b]
        =
        \frac{p^2}{2} \frac{(1-p)^{n-2}}{\mathbb{P}[U_{\le 2}]} {n-2k \choose 2}
        =
        g_{n,p,k}^{mir}.
    \end{equation*}

    Part $b)$ follows very similarly. Let $b \in M_n\langle k \rangle$. Then $b(\overline{ij}) \in M_n\langle k+1 \rangle$ iff the vertices $i^-$ and $j^-$ in $b$ lie on NS paths. There are $(\frac{n}{2}-k)^2$ such pairs. So,
    \begin{equation*}
        \mathbb{P}[X_{mat}=(\overline{ij}), \ i,j \ \mathrm{on \ NS \ paths \ in} \ b]
        =
        \frac{p^2(1-p)^{n-2}}{\mathbb{P}[U_{\le 2}]} (\frac{n}{2}-k)^2
        =
        g_{n,p,k}^{mat}.
    \end{equation*}
\end{proof}

Let $*$ denote $mir$ or $mat$. Let $w_k^*$ = $W_{k+1}^* - W_{k}^*$; this is the number of streets you have to wait between the $k^{th}$ and the $k+1^{th}$ bar being added to the random diagram. Lemma \ref{lem:3:add1bar} shows that $w_k^*$ is a geometric random variable, with parameter $g_{n,p,k}^*$. Note that $W_{\frac{n}{2}}^* = \sum_{k=0}^{\frac{n}{2} -1} w_k^*$. The next theorem bounds the probability that $W_{\frac{n}{2}}^*$ is large.

\begin{thm}\label{thm:1:mainthmfor<2obs:mhattan}
    Let $*$ represent $mat$ or $mir$. Let $p\le Cn^{-1}$, $C$ a constant. Then for all $\alpha>0$, 
    \begin{equation*}
        \mathbb{P}[W^{*}_{\frac{n}{2}} \ge \alpha p^{-2}] 
        \le 
        A_{*}e^{-\frac{1}{4C_2} \alpha},
    \end{equation*}
    where $A_{mir} = cosh(\pi)$, and $A_{mat} = \frac{sinh(\pi)}{\pi}$, and $C_2 = \frac{1}{2}C^2 + C +1$.
\end{thm}

\begin{proof}[Proof of Theorem \ref{thm:1:mainthmfor<2obs:mhattan}]
Let's look at the Manhattan case. We first note that, using $p < Cn^{-1}$, 
\begin{equation*}
    \begin{split}
        g_{n,p,k}^{mat} 
        &= 
        \frac{(\frac{n}{2}-k)^2}{\frac{1}{2}n^2 - \frac{1}{2}n + np^{-1} -n + p^{-2} - 2p^{-1} + 1}\\
        &\ge
        \frac{(\frac{n}{2}-k)^2}{(\frac{1}{2}C^2 + C +1)p^{-2}} \\
        &= 
        (\frac{n}{2} - k)^2 C_2^{-1}p^2\\
    \end{split}
\end{equation*}
Now, recall that $W_{\frac{n}{2}}^{mat} = \sum_{k=0}^{\frac{n}{2} -1} w_k^{mat}$, that $w_k^{mat}$ are independent and geometrically distributed with parameter $g_{n,p,k}^{mat}$. Recall also that the moment generating function of a geometric random variable $G$ with parameter $\lambda$ is given by
\begin{equation*}
    \mathbb{E}[e^{tG}] = \frac{\lambda}{1-(1-\lambda)e^t},
\end{equation*}
for $t < -\log(1-\lambda)$. Setting $t = \frac{p^2}{4C_2}$ and $\lambda = p^2$, this inequality holds, since $C_2 = \frac{1}{2}C^2 + C +1 \ge \frac{1}{2}$. We have, using Chebyshev's exponential inequality,
\begin{equation*}
    \begin{split}
        \mathbb{P}[W_{\frac{n}{2}}^{mat} \ge \alpha p^{-2}]
        &\le 
        e^{- \frac{1}{4C_2} \alpha} \mathbb{E}\left[ e^{\frac{p^2}{4C_2} W_{\frac{n}{2}}^{mat}} \right]\\
        &=
        e^{- \frac{1}{4C_2} \alpha}
        \prod_{k=0}^{\frac{n}{2} -1} \left(
        \frac{g_{n,p,k}^{mat}}{1-(1- g_{n,p,k}^{mat})e^{\frac{p^2}{4C_2}}}
        \right)\\
        & = 
        e^{- \frac{1}{4C_2} \alpha}
        \prod_{k=0}^{\frac{n}{2} -1} \left( 1+ 
        \frac{e^{\frac{p^2}{4C_2}} -1 }{1-(1- g_{n,p,k}^{mat})e^{\frac{p^2}{4C_2}}}
        \right).
    \end{split}
\end{equation*}
Using $e^t \le 2t +1$, (which holds for $t = \frac{p^2}{4C_2} <1 $, which in turn always holds, since $C_2 \ge \frac{1}{2}$), we have
\begin{equation*}
    \begin{split}
        1-(1- g_{n,p,k}^{mat})e^{\frac{p^2}{4C_2}}
        & \ge
        1- (1- (\frac{n}{2} - k)^2 C_2^{-1}p^2) (\frac{1}{2} C_2^{-1}p^2 +1)\\
        & \ge
        C_{2}^{-1}p^2 ((\frac{n}{2} - k)^2 - \frac{1}{2}),
    \end{split}
\end{equation*}
which gives:
\begin{equation*}
    \begin{split}
        \mathbb{P}[W_{\frac{n}{2}}^{mat} \ge \alpha p^{-2}]
        &\le 
        e^{- \frac{1}{4C_2} \alpha}
        \prod_{k=0}^{\frac{n}{2} -1} \left( 1+ 
        \frac{ \frac{1}{2} C_2^{-1}p^2 }{ C_{2}^{-1}p^2 ((\frac{n}{2} - k)^2 - \frac{1}{2}) }
        \right)\\
        & =
        e^{- \frac{1}{4C_2} \alpha}
        \prod_{k=1}^{\frac{n}{2}} \left( 1+ 
        \frac{1}{ 2k^2 - 1 }
        \right)\\
        & \le 
        e^{- \frac{1}{4C_2} \alpha}
        \prod_{k=1}^{\frac{n}{2}} \left( 1+ 
        \frac{1}{ k^2 }
        \right)\\
        & \le 
        \frac{\sinh{\pi}}{\pi}
        e^{- \frac{1}{4C_2} \alpha}
    \end{split}
\end{equation*}
as desired. In the last inequality we used the product formula  $\sin{\pi z} = \pi z \prod_{\nu = 1}^{\infty} (1-\frac{z^2}{\nu^2})$, with $z=i$. 

The Mirror model case is almost identical; all the above working is the same except the expression $(\frac{n}{2} -k)^2$ is replaced with $\frac{1}{2}{n-2k \choose 2}$. This gives
\begin{equation*}
    \begin{split}
        \mathbb{P}[W_{\frac{n}{2}}^{mat} \ge \alpha p^{-2}]
        & \le 
        e^{- \frac{1}{4C_2} \alpha}
        \prod_{k=0}^{\frac{n}{2}-1 } \left( 1+ 
        \frac{ \frac{1}{2}}{ \frac{1}{2}{n-2k \choose 2} -  \frac{1}{2} }
        \right)\\
        & \le 
        e^{- \frac{1}{4C_2} \alpha}
        \prod_{k=1}^{\frac{n}{2} } \left( 1+ 
        \frac{ 4 }{ 2k(2k-1)}
        \right)\\
        &\le 
        \cosh(\pi) e^{- \frac{1}{4C_2} \alpha}
    \end{split}
\end{equation*}
as desired, where for the last equality we used the product formula $\cos(\pi z) = \prod_{\nu=1}^{\infty} (1- \frac{4z^2}{(2\nu -1)^2})$, with $z=i$. \\
\end{proof}

We can now compare the full models with the models assuming at most two mirror per street. Let $i \in \mathbb{N}$. Let $\tau(i)$ be the random variable given by the number of the first $i$ streets which have at most 2 mirrors. We see that $\tau(i)$ is binomially distributed with parameters $(i, \mathbb{P}[U_{\le 2}])$. Essentially what we would like to say is that if we omit each street which has more than 2 mirrors, we don't, in distribution, add any bars. 

This sounds like it should follow from the remark (\ref{rmk:multdoesntremovebars}), but it is more subtle. Let us illustrate why: certainly if $ab$ has $k$ bars, then we can conclude that each of $a$ and $b$ have no more than $k$ bars. However, if $abc$ has $k$ bars, it is very possible that $ac$ has more than $k$ bars. So, when removing factors from the middle of a product, there is more to be proved.

\begin{lem}\label{lem:5&6:comparingVkandWk}
    Let $*$ denote $mir$ or $mat$. Then $\mathbb{P}[V_k^* \le i] \ge \mathbb{P}[W_k^* \le \tau(i)]$.
\end{lem}

Recall that $b \in B_n^k$ iff $b$ has at least $k$ bars, and $b \in M_n^k$ similar. Note that $V_k^{mir} \le i$ iff $Z_{mir}^{[i]} \in B_n^k$; similar for $V_k^{mat}$ and $W_k^*$. So Lemma \ref{lem:5&6:comparingVkandWk} can be rewritten as:
\begin{equation}\label{eqn:lem5&6rewritten}
    \mathbb{P}[Z_{mir}^{[i]} \in B_n^k] \ge \mathbb{P}[X_{mir}^{\tau(i)} \in B_n^k],
\end{equation}
and similar for Manhattan. 

We postpone the proof of Lemma \ref{lem:5&6:comparingVkandWk}, and first see how it is implemented, combining with Theorem \ref{thm:1:mainthmfor<2obs:mhattan} in proving part $a)$ of Theorem \ref{thm:1.1mainthm}.

\begin{proof}[Proof of part $a)$ of Theorem \ref{thm:1.1mainthm}]
     We approximate $Z_*^{[i]}$ with $X_*^{\tau(i)}$, that is, we approximate by ignoring streets which have more than two mirrors. Since the expected number of mirrors per street is at most $C$, we expect (at least for $C$ small) the proportion of streets with at most two mirrors to be large. Indeed:
    \begin{equation*}
        \begin{split}
            \lim_{n \to \infty} \mathbb{P}[U_{\le 2}] 
            &= 
            \lim_{n \to \infty} (1-p)^{n-2} \left( 
            (1-p)^2 + np(1-p) + {n \choose 2}p^2
            \right)\\
            & \ge
            \lim_{n \to \infty} (1-p)^{cp^{-1}-2}
            \left( (1-p)(1-p+C) + \frac{C}{2}(C-p) \right)
        \end{split}
    \end{equation*}
    the limit of which is $(\frac{1}{e})^C(1+C+\frac{C^2}{2}) =: C_3$. We can pick $n \in \mathbb{N}$ such that $\mathbb{P}[U_{\le 2}] > \frac{C_3}{2}$. Recalling $\tau(i)$ is binomially distributed with parameters $(i, \mathbb{P}[U_{\le 2}])$, by Hoeffding's inequality,
    \begin{equation}\label{eqn:hoeffding}
        \mathbb{P} \left[ \tau(i) \le \frac{C_3 i}{2} \right] 
        \le 
        \exp \left[ -2i \left( \mathbb{P}[U_{\le 2}] - \frac{C_3}{2} \right)^2 \right]
        \le 
        \exp \left[ -2i \left( \frac{C_3}{4} \right)^2 \right]
    \end{equation}
    for $p$ small enough.
    
    Let $i = \alpha p^{-2}$. Now using Lemma \ref{lem:5&6:comparingVkandWk},
    \begin{equation*}
        \begin{split}
            \mathbb{P}[V_k^* \ge \alpha p^{-2}] 
            &\le 
            \mathbb{P}[W_k^* \ge \tau(\alpha p^{-2})]\\
            &\le
            \mathbb{P} \left[ W_k^* \ge \frac{C_3}{2}\alpha p^{-2} \right]
            +
            \mathbb{P} \left[ \tau(i) \le \frac{1}{2}\alpha p^{-2} \right]\\
            &\le
            A_* \exp\left[ -\frac{C_3}{8 C_2}\alpha \right]
            + 
            \exp\left[ -2\alpha p^{-2} \left(\frac{C_3}{4} \right)^2 \right] \\
            &\le 
            2A_* \exp \left[ -\frac{1}{8e^C} \alpha \right]
        \end{split}
    \end{equation*}
    where the second to last inequality is from Theorem \ref{thm:1:mainthmfor<2obs:mhattan} and equation (\ref{eqn:hoeffding}), and the last is for $p$ small enough.
\end{proof}

It remains to prove Lemma \ref{lem:5&6:comparingVkandWk}. 

\begin{proof}[Proof of Lemma \ref{lem:5&6:comparingVkandWk}]
    We prove equation (\ref{eqn:lem5&6rewritten}); an identical proof holds for the Manhattan case. We work by induction on $i$ and $k$. The equation is trivially true for $k=0$, since both sides are equal to 1, and for $i=1$, since if $\tau(1)=1$, then $Z^{[1]}_{mir} = X^1_{mir}$, and otherwise, $X_{mir}^0 = \mathrm{id}$. 
    
    Assume the Lemma holds for $n-1$ and $k-1$. The left hand side of equation (\ref{eqn:lem5&6rewritten}) is:
    \begin{equation*}
        \begin{split}
            \mathbb{P}[Z_{mir}^{[i]} \in B_n^k] 
            &= 
            \mathbb{P} \left[ Z_{mir}^{[i-1]} Z_{mir}^{(i)} \in B_n^k \ | \ U_{>2}^{(i)} \right] \cdot \left( 1- \mathbb{P}[U_{\le2}] \right)\\
            &\hspace{1cm} + 
            \mathbb{P} \left[ Z_{mir}^{[i-1]} Z_{mir}^{(i)} \in B_n^k \ | \ U_{\le2}^{(i)} \right] \cdot \mathbb{P}[U_{\le2}]\\
            &\ge 
            \mathbb{P} \left[ Z_{mir}^{[i-1]} \in B_n^k \ | \ U_{>2}^{(i)} \right] 
            \cdot \left( 1- \mathbb{P}[U_{\le2}] \right)\\
            &\hspace{1cm} + 
            \mathbb{P} \left[ Z_{mir}^{[i-1]} X_{mir}^{(i)} \in B_n^k \ | \ U_{\le2}^{(i)} \right] \cdot \mathbb{P}[U_{\le2}]
        \end{split}
    \end{equation*}
    where we have noted that the number of bars in $Z_{mir}^{[i-1]} Z_{mir}^{(i)}$ cannot be less that in $Z_{mir}^{[i-1]}$, and that $Z_{mir}^{(i)}$ is equal to $X_{mir}$ when conditioned on $U_{\le2}^{(i)}$. Now the above is at least:
    \begin{equation*}
        \begin{split}
            &\ge 
            \mathbb{P} \left[ X_{mir}^{\tau(i-1)} \in B_n^k \ | \ U_{>2}^{(i)} \right] \cdot \left( 1- \mathbb{P}[U_{\le2}] \right)
            + 
            \mathbb{P}[ X_{mir}^{\tau(i)} \in B_n^k \ | \ U_{\le2}^{(i)} ] \cdot \mathbb{P}[U_{\le2}]\\
            &=
            \mathbb{P}[X_{mir}^{\tau(i)} \in B_n^k]
        \end{split}
    \end{equation*}
    where we used the inductive assumption, the fact that under $U_{> 2}^{(i)}$, $\tau(i) = \tau(i-1)$, and the following claim. The proof of the claim therefore concludes the whole proof.
    \begin{lem}
        We have that $\mathbb{P} \left[ Z_{mir}^{[i-1]} X_{mir}^{(i)} \in B_n^k \ | \ U_{\le2}^{(i)} \right]
        \ge 
        \mathbb{P}[ X_{mir}^{\tau(i)} \in B_n^k \ | \ U_{\le2}^{(i)} ]$
    \end{lem}
    \begin{proof}\renewcommand{\qedsymbol}{}
        To prove the claim, we split the left hand term based on whether or not $X_{mir}^{(i)}$ adds a bar to $Z_{mir}^{[i-1]}$:
        \begin{equation*}
            \begin{split}
                \mathrm{LHS} 
                &=
                \mathbb{P} \left[ Z_{mir}^{[i-1]} \in B_n^{k-1} \ | \ U_{\le2}^{(i)} \right]\cdot g_{n,p,k}^{mir} 
                +
                \mathbb{P} \left[ Z_{mir}^{[i-1]} \in B_n^{k} \ | \ U_{\le2}^{(i)} \right]\\
                &\ge 
                \mathbb{P} \left[ X_{mir}^{\tau(i-1)} \in B_n^{k-1} \ | \ U_{\le2}^{(i)} \right]\cdot g_{n,p,k}^{mir} 
                +
                \mathbb{P} \left[ X_{mir}^{\tau(i-1)} \in B_n^{k} \ | \ U_{\le2}^{(i)} \right]\\
                &=
                \mathbb{P} \left[ X_{mir}^{\tau(i-1)} X_{mir}^{(i)} \in B_n^k \ | \ U_{\le2}^{(i)} \right]
            \end{split}
        \end{equation*}
        and now recalling that under $U_{\le 2}$, $\tau(i) = \tau(i-1) +1$, the result follows. This concludes the proof of Lemma \ref{lem:5&6:comparingVkandWk} and part $a)$ of Theorem \ref{thm:1.1mainthm}.
    \end{proof}

\end{proof}

\section{Acknowledgements}
I would like to thank Sasha Sodin for many useful discussions.

\newpage

\bibliography{bibx}
\bibliographystyle{plain}
\end{document}